\documentclass[]{interact}

\usepackage{lmodern}

\usepackage{hyperref}

\numberwithin{equation}{section}

\newcommand{\R}{\mathbf R}
\newcommand{\F}{\mathcal F}
\newcommand{\Y}{\mathcal Y}
\newcommand{\B}{{\mathcal B}}
\renewcommand{\H}{\mathcal H}
\renewcommand{\L}{\mathcal L}
\newcommand{\inner}[2]{\left\langle{#1},{#2}\right\rangle}
\newcommand{\inter}{\cap}
\renewcommand{\Re}{\mathop{Re}}

\newtheorem{theorem}{Theorem}[section]
\newtheorem{lemma}[theorem]{Lemma}

\begin{document}


\title{The Adaptive LQ Regulator}  

\author{
	\name{Omar Hijab
	}
	\affil{Department of Mathematics, Temple University,  Philadelphia, PA.}
	\thanks{E-mail: hijab@temple.edu}
}

\maketitle    

\begin{abstract}                         
The optimal adaptive control of  a linear system in a signal-plus-noise setting with infinite horizon LQ regulator cost is studied. The class of partially observed linear systems for which the certainty equivalence property holds is identified.  It is also shown that a linear system is adaptively stabilizable if and only if it is uniformly stabilizable, and the class of partially observed linear systems for which the certainty equivalence value function is a supersolution of the Bellman equation is identified.
\end{abstract}


\begin{amscode}
93E20, 60G35, 93E35, 34H15, 49L25
\end{amscode}

\section{Introduction}
\label{intro}

Adaptive control \cite{AK, BGW, GS, IS, KKK, KV, NA, SB, T} is the control of systems in the presence of system parameter uncertainty, which we assume is the result of observations $y$ of a signal $z$ corrupted by noise $w$, in a signal-plus-noise setting.

Adaptive control problems may be formulated in a non-Bayesian context, with no prior distribution on the unknown system parameter $\theta$ or the noise $w$, or in a Bayesian context, where there is a prior distribution. 

In the former case, one designs control laws using a maximum likelihood estimate $\hat\theta(t)$  given past observations $y(s)$, $0\le s\le t$. In the latter case,  one designs control laws using the conditional probabilities
\begin{equation}\label{condprob}
p_j(t)=P(\theta=j \mid y(s),0\le s\le t)
\end{equation}
that the system is in regime $j$, based on past observations. These are computed via Bayes' rule, which then expresses the conditional expectation $\hat\theta(t)$  as an average likelihood estimate.

A certainty equivalent control law is one designed in two stages: resolving regime uncertainty followed by the application of a control law corresponding to a known regime. In the former case, this means implementing the optimal control law corresponding to the parameter estimate $\hat\theta(t)$, and, in the latter case, implementing the optimal control law corresponding to a given parameter, averaged over all parameter values. The certainty equivalence property is the fortuitous occurence when the certainty equivalent control law turns out to be optimal, or stable, or both, for the uncertain system.

The adaptive LQ regulator, a generalization of the LQ regulator \cite{AM, DPB, RWB, CTC, REK, S}, is an optimal control problem with linear dynamics and quadratic cost, where the system matrices, collectively denoted $\theta$, are  not known to the controller, and thus the state $x(t)$ is partially observed. 
 
Any partially observed optimal control problem is equivalent \cite{MHAD, FP, KO, PR} to a completely observed optimal control problem with a much larger state space. For the adaptive LQ regulator,  the equivalent completely observed problem has nonlinear dynamics and non-quadratic cost.

We assume there are only a finite number $N$ of  operating regimes. Then the equivalent completely observed problem has state
\begin{equation}\label{dyna}
(x_1(t),\dots,x_N(t),p_1(t),\dots,p_N(t)),
\end{equation}
where $x_j(t)$ is the system state in regime $j$, $p_j(t)$ is as above, and optimal controls are determined by dynamic programming, which in the current setting reduces to the Bellman differential equation.

In this paper, we isolate a class of linear systems for which the adaptive LQ regulator is completely and explicitly solvable, and where the certainty equivalence property holds. 
We also clarify the link between adaptive stabilizability, uniform stabilizability, and supersolutions of the Bellman equation.

By introducing an entropy factor into the cost functional, we identify the class of linear systems for which the certainty equivalence property holds. Since minimizing entropy is equivalent to maximizing information, introducing this factor is natural. In fact, this introduction of the entropy is a quantitative expression of dual control \cite{YBS, F}, wherein the controller attempts to  directly affect both the system state and the system uncertainty. 
The fact that the factor then facilitates the explicit solvability of the Bellman  equation is the surprise. This approach is then extended to supersolutions of the Bellman equation, leading to stabilizing feedback controls. 

That the standard LQ regulator has a linear optimal feedback law is a consequence of the fact that the Bellman equation of the LQ regulator is explicitly solved by the quadratic function $\inner{Kx}{x}/2$, where the Kalman gain $K$ is given by the Ricatti equation.
 
In the same way, ultimately our analysis rests on explicit solutions of the Bellman equation of the adaptive LQ regulator, through the above quadratic function and the simple calculus fact
\begin{equation}\label{calculusfact}
H(p)=-\sum_{j=1}^N p_j\log p_j\quad\Rightarrow\quad -\frac{\partial^2H}{\partial p_j\partial p_k}=\frac{\delta_{jk}}{p_j}.
\end{equation}
Nevertheless, partially observed problems are different enough from completely observed problems to warrant a careful formulation and re-working of the optimal control problem in the partially observed setting.

Another generalization of the LQ regulator is considered in \cite{AJK}. Here the system dynamics matrices are perturbed by additive  noise, resulting  in a completely observed optimal control problem whose Bellman equation can be solved by a quadratic function given by an extended Ricatti equation.

In \cite{OH}, the entropy was introduced to obtain explicit solutions to the Bellman equation, but the class of linear systems for which the certainty equivalence property holds was not identified, and supersolutions were not considered. 
Moreover the formulation there focused on the sample space of $(\theta,w)$, where $w$ is the observation noise. The approach here, with the focus on the sample space of $(\theta,y)$, simplifies the derivation of the filtering equations and expresses the problem in its natural setting, leading to optimal results.

\section{Overview}
\label{overview}

Let $(A_j,B_j,C_j)$, $j=1,\dots,N$, be minimal (\S\ref{lqr}) linear systems and let  $x_j$, $j=1,\dots,N$, be initial states. Let $S^N$ denote the set of probabilities $p=(p_1,\dots,p_N)$ on $\{1,\dots,N\}$, let
$p$ be in $S^N$, and let $\theta$ be a random variable valued in $\{1,\dots,N\}$  and distributed according to $p$.

Let $\R^i$ and $\R^o$ be the euclidean spaces where inputs $u$ and outputs $y$ take their values. We do not specify explicitly the euclidean space where the state $x$ takes its values, and the state space dimension may vary with $j$. 

Throughout, $\inner{\upsilon}{w}$ denotes the dot product of vectors $\upsilon,w$ in euclidean space, $|\upsilon|^2=\inner{\upsilon}{\upsilon}$ denotes the length squared, $M^*$ is the transpose of $M$, $|M|$ is the norm of $M$, and $\Re M =(M+M^*)/2$.

The system dynamics are
\begin{equation}\label{stateeq1}
\dot x(t) =  A_\theta x(t)+B_\theta u(t),\qquad x(0)=x_\theta,
\end{equation}
the observations are given by the signal-plus-noise model
\begin{equation}\label{obser1}
y(t) = \int_0^tG_\theta x(s)\,ds+w(t),
\end{equation}
and the cost starting from $(x,p)=(x_1,\dots,x_N,p_1,\dots,p_N)$ corresponding to $u$ is
 \begin{equation}\label{JT1}
J(x,p,u)=\frac12E^P\left(\int_0^\infty \left(|C_\theta x|^2+|u(t)|^2\right)\,dt\right).
\end{equation}
Here $z(t)=G_\theta x(t)\in\R^o$ is the signal, the noise $w$ is a Wiener process independent of $\theta$,  and $E^P$ is the expectation against the underlying statistics $P$.  We assume the system is  {\em partially observed} in the sense that the  control $u(t)\in\R^i$ depends only on past observations $y(s)$, $0\le s\le t$, for each $t\ge0$, and we seek controls minimizing $J$. This is in contrast with the completely observed case where $u(t)$ is allowed to depend also on $\theta$. This partially observed optimal control problem is the adaptive LQ regulator.

Let $K=K(A,B,C)$ be the unique positive definite solution of the Ricatti equation
\begin{equation}\label{algricatti}
0=C^*C+A^*K+KA-KBB^*K
\end{equation}
corresponding to  the minimal triple $(A,B,C)$.
When $\theta\equiv j$, the partially observed problem is the standard LQ regulator. Here the optimal cost starting from $x_j$ is $U_j(x_j)=\inner{K_jx_j}{x_j}/2$ where $K_j=K(A_j,B_j,C_j)$,  $\bar A_j=A_j-B_jB^*_jK_j$ is stable, and the optimal control is given by the feedback $u(t)=-B^*_jK_jx(t)$. It follows that, if we allow controls in \eqref{stateeq1}, \eqref{obser1}, \eqref{JT1} 
to depend also on $\theta$, the optimal cost is
$$U_{ce}(x,p)=\frac12 \sum_{j=1}^N p_j \inner{K_jx_j}{x_j}=\frac12E^P(\inner{K_\theta x_\theta}{x_\theta}),$$
and the optimal control satisfies
\begin{equation}\label{compobs}
u(t)=-B^*_\theta K_\theta x(t).
\end{equation}

For the partially observed problem, things can go wrong even in the simplest cases. For example, suppose $\theta=\pm1$ and consider the scalar integrator
$$
\dot x(t) = \theta u(t),\qquad x(0)=x_\theta\in\R, 
$$
with zero signal, $G_\theta\equiv0$. Then we may restrict to deterministic controls (see \eqref{determ} below) and the cost $J(x_+,x_-,p,u)$ reduces to
$$\frac12\int_0^\infty (px_+(t)^2+(1-p)x_-(t)^2+u(t)^2)\,dt.$$
Here $x_+=x_+(0),x_-=x_-(0)$ are the two possible initial states and  $p=P(\theta=1)$. Let $U(x_+,x_-,p)=\inf_uJ(x_+,x_-,p,u)$ be the value function. The optimal feedback when $p(1-p)=0$ is $u(t)=-\theta x(t)=-e^{-t}\theta x_\theta$. This control is deterministic iff $\theta x_\theta$ is deterministic.  More generally, if $u$ is any deterministic control with finite cost, $J(x_+,x_-,p,u)<\infty$, then, as we will see, $u$ is stabilizing, $x(\infty)=0$, hence 
$$0=\theta x(\infty)=\theta x_\theta+\int_0^\infty u(t)\,dt,$$ 
which implies $\theta x_\theta$ is deterministic. Thus, if $x_+\not=-x_-$ and $J(x_+,x_-,p,u)<\infty$, then  $p(1-p)=0$. 
If $x_+=-x_-$, then $\theta x_\theta$ is deterministic and $u(t)=-\theta x(t)$ is deterministic. 
Since this $u$ is the optimal LQ regulator feedback, we conclude 
\begin{equation*}
U(x_+,x_-,p) =
\begin{cases}
\frac12(px_+^2+(1-p)x_-^2), & p=0,1\text{ or }x_+=-x_-,\\
 \infty, & p\not=0,1\text{ and }x_+\not=-x_-.
 \end{cases}.
\end{equation*}
This system is strongly not stabilizable.

In the positive direction, we have the following results: The linear system ensembles $(A_j,B_j,C_j,G_j)$, $j=1,\dots,N$,  for which the certainty equivalence property holds for the adaptive LQ regulator are those satisfying
\begin{equation}\label{explicit}
 FG_j=B_j^*K_j,\qquad j=1,\dots,N,
\end{equation}
for some linear feedback $F:\R^o\to\R^i$. 

Given $f=f(x_1,\dots,x_N,p_1,\dots,p_N)$, let  
\begin{equation}\label{gen}
\L f= \frac12\sum_{j,k=1}^N\frac {\partial^2f}{\partial p_j\partial p_k}\,p_jp_k\inner{z_j-\hat z}{z_k-\hat z},
\end{equation}
where $z_j=G_jx_j$, and $\hat z=\sum_{j=1}^Np_jz_j$. Let $D_j=\nabla_jf$ denote the gradient with respect to $x_j$, $D=(D_1,\dots,D_N)$, $\nabla f=(\nabla_1f,\dots,\nabla_Nf)$, 
and let 
\begin{equation*}
\begin{split}
\H(x,p,D)&=\min_{u\in\R^i}\left(\frac12|u|^2+ \sum_{j=1}^N \frac12p_j|C_jx_j|^2 +\inner{D_j}{A_jx_j+B_ju}\right) \\
&=  \frac12\sum_{j=1}^N p_j|C_jx_j|^2 + \sum_{j=1}^N\inner{D_j}{A_jx_j}
- \frac12\left|\sum_{j=1}^NB_j^*D_j\right|^2. 
\end{split}
\end{equation*}
Then the Bellman equation of the adaptive LQ regulator is
\begin{equation}\label{bellman1}
-\L f=\H(x,p,\nabla f),
\end{equation}
in the sense that the optimal cost corresponding to \eqref{stateeq1}, \eqref{obser1}, \eqref{JT1} satisfies \eqref{bellman1}.

Let 
\begin{equation}\label{cep1}
V_{ce}(x,p)=\frac12\sum_{j=1}^Np_j\inner{K_jx_j}{x_j}+ H(p),
\end{equation}
where $H$ is given in \eqref{calculusfact}. Then the validity of \eqref{explicit} for some partial isometry $F:\R^o\to\R^i$ implies $V_{ce}$ is a solution of \eqref{bellman1}. Moreover, under a mild restriction, the converse holds. These last facts are a purely algebraic computation using \eqref{algricatti}, independent of the probabilistic set-up \eqref{stateeq1}, \eqref{obser1}, \eqref{JT1}. The rest of this section consists of the details surrounding these assertions.

Let $(x,p)$ be an initial state ensemble $x=(x_1,\dots,x_N)$ and probability $p=(p_1,\dots,p_N)$.  A control $u$ is  {\em admissible at $(x,p)$} (\S\ref{stabiliz}) if $u(t)$ depends only on past observations $y(s)$, $0\le s\le t$, for each $t\ge0$, and there exists statistics $P$ on the sample space of $(\theta,y)$ under which \eqref{stateeq1}, \eqref{obser1} hold, in the sense that under  $P$,
\begin{equation}\label{L2control}
\int_0^t|u(s)|^2\,ds<\infty,\qquad t\ge0,
\end{equation}
almost surely, $\theta$ is distributed according to $p$, and 
$$w(t)=y(t)-\int_0^tz(s)\,ds,\qquad t\ge0,$$ 
is a Wiener process independent of $\theta$.  

There is no reason for a sufficiently general control $u$ to be admissible: the statistics $P$ may not exist. In \S\ref{pols}, we show for each control $u$ and initial condition $(x,p)$, there is at most one $P$. Subsequently, in \S\ref{stabiliz}, \S\ref{optimal},  we show that for each $(x,p)$ and for the specific feedback controls we seek, $P$ exists. We also show the class of controls admissible at $(x,p)$, $p$ in the interior of $S^N$,  
does not depend on $(x,p)$ (\S\ref{stabiliz}).

Let $P$ be the statistics corresponding to $(x,p)$ and admissible control $u$, and let $\delta_j$ be the distribution on $\{1,\dots,N\}$ equal to one at $j$. Then (\S\ref{pols}) the conditional probability $P_j(A)=P(A \mid \theta=j)$ is the statistics corresponding to $(x,\delta_j)$ and $u$.  If  $J(x,\delta_j,u)$ is the corresponding cost, then $J(x,\delta_j,u)$ depends only on $j$, $x_j$, and $u$, and
\begin{equation}\label{affine}
J(x,p,u)=\sum_{j=1}^N p_j J(x,\delta_j,u),\qquad p\in S^N.
\end{equation}
Let $\bar u_j(t)=E^{P_j}(u(t))$ and let $J_j$ be the cost corresponding to the $j$-th system.
 By the Cauchy-Schwarz inequality, $J(x,\delta_j,u)\ge J_j(x_j,\bar u_j)$, hence
\begin{equation}\label{lambda=0}
J(x,p,u)\ge U_{ce}(x,p).
\end{equation}
This is consistent with the fact that $U_{ce}$ is a subsolution of the Bellman equation \eqref{bellman1}. 

 An admissible control $u$ is {\em stabilizing at $(x,p)$} if  $x(t)\to0$ as $t\to\infty$, almost surely $P$.  We use $J$ to find stabilizing admissible controls. For this approach to succeed, it is necessary that $J$ be Lyapunov in the sense that the finiteness of $J(x,p,u)$ implies $u$ is stabilizing at $(x,p)$. In fact (\S\ref{stabiliz}) this is so, when $(A_j,B_j,C_j)$, $j=1,\dots,N$, are minimal.
 
 An admissible  control $u$ is  {\em optimal at $(x,p)$} if $J(x,p,u)=U(x,p)$, where
\begin{equation}\label{valueU}
U(x,p)=\inf_u J(x,p,u).
\end{equation}
If $u$ is optimal at $(x,p)$ and $U(x,p)$ is finite, it follows that $u$ is stabilizing at $(x,p)$.
If $U$ is finite at some $(x,p_0)$ with $p_0$ in the interior of $S^N$, by \eqref{affine}, $U$ is finite on $\{x\}\times S^N$, and a standard argument implies\footnote{This inequality is 
not used in this paper.}\addtocounter{footnote}{-1}\addtocounter{Hfootnote}{-1}
$$\sum_{j=1}^Np_j\left|\frac{\partial U}{\partial p_j}(x,p)\right|\le U(x,p)$$
almost everywhere on $\{x\}\times S^N$.
  
Given an initial state $x_j$ and an admissible control $u$, let $z_j(t)=G_jx_j(t)$ be the signal when $\theta\equiv j$. Let
\begin{equation}\label{zhat}
\hat z(t)= \sum_{j=1}^N p_j(t)z_j(t),
\end{equation}
where $p_j(t)$ is the conditional probability \eqref{condprob},  $j=1,\dots,N$.
A first guess for a stabilizing admissible control is the certainty equivalent feedback control of the form
\begin{equation}\label{cep}
u(t) = -F\hat z(t),\qquad t\ge0.
\end{equation}
If,  for each initial condition $(x,p)$, there is a unique stabilizing admissible  control $u$ at $(x,p)$ satisfying \eqref{cep},  the system is {\em adaptively stabilizable with feedback $F$}. If $\bar A_j=A_j-B_jFG_j$, $j=1,\dots,N$, are stable,  the system is {\em uniformly stabilizable with feedback $F$}. By specializing $p=\delta_j$, $j=1,\dots,N$, we see uniform stabilizability with feedback $F$ is a necessary condition for  adaptive stabilizability with feedback $F$. We show (\S\ref{stabiliz}) that uniform stabilizability with feedback $F$  implies adaptive stabilizability with feedback $F$, and moreover the admissible $u$ satisfying \eqref{cep} has finite cost $U^F(x,p)<\infty$. 

Note when \eqref{explicit} holds, the feedbacks \eqref{compobs}, \eqref{cep} are consistent: the latter is the conditional expectation of the former.

When there is no signal, $G_\theta\equiv0$, $\bar u_j$ does not depend on $j$, so the partially observed problem reduces to seeking a deterministic control $u$ achieving the minimum in
\begin{equation}\label{determ}
U(x,p)=\inf_u\left(\sum_{j=1}^N p_jJ_j(x_j,u)\right).
\end{equation}
Let $\nabla_j$ denote the gradient with respect to $x_j$, and let $U_j^F(x)$ denote the cost of the $j$-th system corresponding to the  feedback $u=-Fx_j$, $j=1,\dots,N$. 
When there is  no signal, if $\bar A_j=A_j-B_jF$, $j=1,\dots,N$, are stable, differentiating $J(x,p,u)$ and  using \eqref{determ}  and the Cauchy-Schwarz inequality as in \eqref{CSF} implies\footnotemark
$$
\sum_{j=1}^N\left|\nabla_j \sqrt{U(x,p)}\right|\le 
\max_{\substack{|\xi|\le1 \\ 1\le j\le N}}\sqrt{U_j^F(\xi)}
$$
almost everywhere in $x$, for each $p$. 
\setlength{\footnotesep}{4mm}
A similar estimate\footnote{That such an inequality is plausible is suggested by $f=\sum_{j=1}^Np_j|x_j|^2 \Rightarrow \sum_{j=1}^N\left|\nabla_j \sqrt{f}\right|\le1$.} 
in general when there is a signal would be very welcome.

The conditional probabilities $p(t)=(p_1(t),\dots,p_N(t))$ always converge to a limit   $p(\infty)$ as $t\to\infty$. 
One measure of the information contained in the observations is   $-E^P(H(p(\infty)))$, where 
$H$ is the entropy \eqref{calculusfact}.
Then increasing information corresponds to decreasing entropy, which suggests minimizing the modified cost functional $J(x,p,u)+E^P(H(p(\infty)))$  instead.  This strategy of minimizing a sum of an ``operating risk'' and a ``probing risk'' was first discussed for non-Markov models in \cite{F}. As $0\le E^P(H(p(\infty)))\le H(p)$ by concavity of $H$, this modification has no effect on stability: The cost functional is Lyapunov before the modification if and only if it is Lyapunov after the modification. 

Let 
\begin{equation}\label{value}
V(x,p)=\inf_u \left(\vphantom{\int}J(x,p,u)+E^P(H(p(\infty)))\right)
\end{equation}
be the modified optimal cost or value function,
where the infimum is over admissible controls $u$. Then 
$$U(x,p)\le V(x,p)\le U(x,p)+ H(p).$$
Given $(x,p)$, an admissible control $u$ is {\em optimal at $(x,p)$} if
$$V(x,p)=J(x,p,u)+ E^P(H(p(\infty))).$$
 If $u$ is optimal at $(x,p)$ and $V(x,p)$ is finite, then $u$ is stabilizing at $(x,p)$.

In  \S\ref{optimal}, we show a nonnegative $C^2$ solution of the Bellman equation
\eqref{bellman1}
satisfying $f(0,p)=0$ necessarily equals $U$, and a nonnegative $C^2$ solution of \eqref{bellman1} satisfying $f(0,p)=H(p)$ necessarily equals $V$. In general, we conjecture that, when the system is adaptively stabilizable, $f=U$ and $f=V$ are the unique nonnegative viscosity  \cite{CIL} solutions of \eqref{bellman1} with the prescribed values at $x=0$.

A function $f$ is a supersolution of \eqref{bellman1} if the right side is no greater than the left side, for all $(x,p)$. If the reverse inequality holds, $f$ is a subsolution. Supersolutions are of interest because they lead to finite cost and hence stabilizing controls (\S\ref{optimal}). $U_{ce}$ is always a subsolution of the Bellman equation and 
$U_{ce}$ is never a supersolution of the Bellman equation, unless $N=1$ (\S\ref{optimal}).

If the system is adaptively stabilizable with feedback $F$, let $U^F(x,p)=J(x,p,u)$ and $V^F(x,p)=J(x,p,u)+E^P(H(p(\infty)))$ where $u$ satisfies \eqref{cep}. Let $\mu>0$ be such that $\Re \bar A_j=\Re(A_j-B_jFG_j)\le-\mu$ and $1/\mu$ bounds $|B_j|$, $|C_j|$, $|G_j|$, and $|F|$, $j=1,\dots,N$. Then $U^F$ and $V^F$ are viscosity supersolutions satisfying 
\begin{equation}\label{feedbackV}
U^F(x,p)\le V^F(x,p)\le c(\mu) \left(\frac12\sum_{j=1}^Np_j|x_j|^2+H(p)\right),
\end{equation}
for some constant $c(\mu)$ depending only on $\mu$.

Let $V_{ce}$ be given by \eqref{cep1}.
We say {\em certainty equivalence holds} if  $V_{ce}$ is a solution of the Bellman equation \eqref{bellman1}.
We show (\S\ref{optimal}) certainty equivalence holds if there is a partial isometry $F:\R^o\to\R^i$ such that \eqref{explicit} holds.
When this happens, $V=V_{ce}$, the system is adaptively stabilizable with feedback $F$, and for each $(x,p)$, the unique admissible control satisfying \eqref{cep} and stabilizing at $(x,p)$ is the unique admissible control optimal at $(x,p)$. 

We say the signal is {\em faithful} if for some $1\le j\le N$, the linear map $G_j$ is surjective. We show, when the signal is faithful and $N\ge2$, certainty equivalence holds only if  there is a partial isometry $F:\R^o\to\R^i$ such that \eqref{explicit} holds.

We show $V_{ce}$ is a supersolution of \eqref{bellman1} if there is a feedback $F:\R^o\to\R^i$ of norm at most one such that 
 \eqref{explicit} holds. When this happens, the system is adaptively stabilizable with feedback $F$, and for each $(x,p)$ the unique admissible control $u(t)$ given by \eqref{cep} satisfies
\begin{equation}\label{VV}
V(x,p)\le J(x,p,u)+E^P(H(p(t)))\le V_{ce}(x,p).
\end{equation}
We show, when $N\ge2$ and the signal is faithful, $V_{ce}$ is a supersolution only if there is a feedback $F:\R^o\to\R^i$ of norm at most one such that \eqref{explicit} holds (\S\ref{optimal}).

By rescaling $\lambda^2V(x/\lambda,p)$, or by replacing $H$ by $\lambda^2H$,  the same results are valid when the norm of the feedback is at most $\lambda$, hence for any feedback $F$ satisfying \eqref{explicit} and any $(x,p)$, \eqref{cep} yields  an admissible control stabilizing at $(x,p)$.

In principle, the results of this paper should remain valid if the prior distribution $p=(p_1,\dots,p_N)$ of $\theta$ is replaced by a continuous  prior distribution $p$ on the space of all linear systems $(A,B,C,G)$. 

More generally, the definitions of {\em signal process} and {\em statistics} and the entropy identity \eqref{meansqvar} in \S\ref{pols}, \S\ref{filtering}  should extend to the general signal-plus-noise setting where $\theta$ is valued in a Polish space.

\section{The LQ Regulator}
\label{lqr}

We review the LQ regulator \cite{RWB, OH} emphasizing aspects that we use in the adaptive setting. 

The problem is to determine the control $u(t)\in\R^i$ minimizing the cost 
 \begin{equation}\label{JT}
J_T(x,u)=\frac12\int_0^T \left(|Cx(t)|^2+|u(t)|^2\right)\,dt,
\end{equation}
over all (possibly discontinuous)  controls $u$, where the state $x(t)$ is given by
\begin{equation}\label{stateeq}
\dot x(t) =  A x(t)+B u(t),\qquad x(0)=x.
\end{equation}

This is the finite or infinite horizon problem according to whether $T<\infty$ or $T=\infty$. In either case, we take as admissible controls any function $u$ on $[0,\infty)$ valued in $\R^i$ satisfying \eqref{L2control}. With standard convergence on $(T,x)$ and  weak convergence on $u$, $J$ is continuous in $(T,x)$ and lower-semicontinuous in $(T,x,u)$.  It follows that minimizing controls for \eqref{JT} exist, for any starting state $x$ and horizon $T\le\infty$. Let 
\begin{equation}\label{valuefn}
U_T(x)=\min_u J_T(x,u)
\end{equation}
be the optimal cost or value function. Then $U$ is lower-semicontinuous and $U(0,x)=0$. 
Since $J_T(\lambda x,\lambda u)=\lambda^2 J_T(x,u)$, we have
$$U_T(\lambda x)=\lambda^2U_T(x),\qquad \lambda>0.$$
Moreover $U_T(x)$ is increasing as a function of $T$ and nonnegative, and  satisfies the dynamic programming property \cite{FS}
\begin{equation}\label{dp}
U_T(x) = \min_u\left(\frac12\int_0^t\left(|Cx|^2+|u|^2\right)\,ds
+ U_{T-t}(x(t))\right),
\end{equation}
for $0\le t\le T$.

We assume throughout $(A,B)$ is  {\em controllable}, for each $x$, there is a control $u$  driving \eqref{stateeq} to the origin in finite time, and $(A,C)$ is  {\em observable}, the zero-response output map $x\mapsto Ce^{tA}x$, $t\ge0$, determines the initial state $x$.  In short, we assume the system  $(A,B,C)$ is  {\em minimal.} As is well-known \cite{RWB, OH}, this happens iff $(B,AB,A^2B,\dots)$ and $(C/CA/CA^2/\dots)$ have full rank.

Let $U(x)\equiv U_\infty(x)$ and $J(x,u)\equiv J_\infty(x,u)$ be the infinite horizon value function and cost function. Controllability implies the value function $U(x)$ is finite for all $x$, and observability implies $U_T(x)>0$ for all $T>0$ and $x\not=0$.
It follows that $U(x)$ is proper, $|x|\to\infty$ implies $U(x)\to+\infty$. 
 Then $U(x(t))\to0$ as $t\to\infty$ implies $x(t)\to0$ as $t\to\infty$.

By \eqref{dp} with $T=\infty$, for any admissible control $u$,
$$J(x,u)\ge \frac12\int_0^T\left(|Cx|^2+|u|^2\right)\,dt + U(x(T))$$
for $T\ge0$. When $J(x,u)<\infty$, this implies $U(x(t))\to0$ hence $x(t)\to0$, as $t\to\infty$.

Let $\nabla_\xi J$ denote the gradient of $J$ with respect to $x$ in the direction $\xi$, and let $U^0_T(x)$ denote the cost corresponding to $u=0$.  Since $J$ is quadratic in the state trajectory, by the Cauchy-Schwarz inequality,
\begin{equation}\label{CS}
\left|\nabla_\xi J_T(x,u)\right|^2\le  4\,U^0_T(\xi)\,J_T(x,u)
\end{equation}
which implies
$$
\left|\nabla_\xi \sqrt{J_T(x,u)}\right|\le \sqrt{U^0_T(\xi)}.
$$
It follows that 
$$
\left|\sqrt{U_T(x_1)}-\sqrt{U_T(x_2)}\right|\le \left(\max_{|x|\le1}\sqrt{U^0_T(x)}\right)\cdot|x_1-x_2|,
$$
which implies $U$ is continuous as a function of $(T,x)$.  From this and \eqref{dp}, $U=U_T(x)$ is the unique nonnegative solution of  the Bellman equation 
\begin{equation}\label{bellman}
\frac{\partial U}{\partial t}=\frac12|Cx|^2+\inner{\nabla U}{Ax}
-\frac12|\nabla U B|^2,
\end{equation}
with initial condition $U(0,x)=0$, in the viscosity sense \cite{CIL}.

If $K(t)$, $0\le t\le T$, satisfies the Ricatti differential equation
\begin{equation}\label{ricatti}
\frac{\partial K}{\partial t}=C^*C+A^*K+KA-KBB^*K,
\end{equation}
with initial condition $K(0)=0$, then  $U=\inner{K(t)x}{x}/2$, $0\le t\le T$, satisfies \eqref{bellman}, hence  
\begin{equation}\label{ricV}
U_T(x)=\frac12\inner{K(T)x}{x}.
\end{equation}
It follows that $K(t)$ is an increasing positive-definite matrix-valued function of $t$. This implies the unique solution $K(t)$ of \eqref{ricatti} exists and \eqref{ricV} holds, for all $t\ge0$. Combining \eqref{dp}, \eqref{bellman}, and \eqref{ricV}, we obtain for each horizon $T>0$ the unique optimal control in linear feedback form
$$u(t)=-B^*\nabla U_{T-t}(x(t))=-B^*K(T-t)x(t),$$
for $0\le t\le T$.

Since $U(x)$ is finite everywhere,  $K(T)$ increases as $T\to\infty$ to some $K$  satisfying the Ricatti algebraic equation \eqref{algricatti}, hence $\inner{Kx}{x}/2$ satisfies
\begin{equation}\label{steadybellman}
0=\frac12|Cx|^2+\inner{\nabla U}{Ax}
-\frac12|\nabla U B|^2.
\end{equation}
This implies (Theorem \ref{super} with $N=1$) $U=\inner{Kx}{x}/2$ and the unique optimal control is
$$u(t)=-B^*\nabla U(x(t))=-B^*Kx(t),$$
which implies \eqref{algricatti} has a unique positive definite solution $K=K(A,B,C)$ corresponding to every minimal triple $(A,B,C)$. 

It follows from \eqref{dp} that
\begin{equation}\label{estimateK}
\frac12\inner{K(T-t)x(t)}{x(t)}\le \frac12\int_t^T \left(|Cx|^2+|u|^2\right)\,ds,
\end{equation}
for all admissible controls $u$. Let $K(1)>0$ denote the solution of \eqref{ricatti} at $t=1$. Then inserting $T=t+1$ into \eqref{estimateK}  and integrating over $t\ge0$ yields
\begin{equation}\label{estimateK1}
\frac12\int_0^\infty \inner{K(1)x(t)}{x(t)}\,dt \le J(x,u)
\end{equation}
for all admissible controls $u$.

Letting $\bar A=A-BB^*K$ denote the feedback dynamics. Since $J(x,u)=V(x)<\infty$,
$x(t)\to0$ as $t\to\infty$, hence $\bar A$ is stable. 

Assume the system is stabilizable, i.e. there is a feedback $F$ such that $\bar A=A-BF$ is stable. Then the map $u(t)\mapsto u(t)-Fx(t)$, where $x(t)$ satisfies \eqref{stateeq}, is a bijection of the class of all admissible controls. Thus
$$U(x)=\min_u\left(\frac12\int_0^\infty \left(|Cx(t)|^2+|u(t)-Fx(t)|^2\right)\,dt\right),$$
where $x(t)$ satisfies
$$\dot x(t)=Ax(t)+B(u(t)-Fx(t))=\bar Ax(t)+Bu(t).$$
If $U^F(x)$ denotes the cost $J(x,u)$ corresponding to the feedback $u=-Fx$ in  \eqref{stateeq}, then repeating the logic leading to \eqref{CS} yields
\begin{equation}\label{CSF}
\left|\nabla \sqrt{U(x)}\right|\le \max_{|\xi|\le1}\sqrt{U^F(\xi)},
\end{equation}
valid for any control $u$.

\section{Partially Observed Systems}
\label{pols}

We derive the existence and uniqueness results for the underlying statistics $P$ for the signal-plus-noise model
\begin{equation}\label{s+n}
y(t)=\int_0^tz(s)\,ds+w(t),\qquad t\ge0.
\end{equation}
Here $w(t)$ is a Wiener process independent of $\theta\in\{1,\dots,N\}$,  $z_j(t)$, $j=1,\dots,N$, are processes depending only on the past of the observations $y(t)$, and $z(t)=z_\theta(t)$. 

Let $C([0,\infty),\R^o)$ be the space of all continuous sample paths $\alpha$ on $[0,\infty)$ valued in $\R^o$.  Then  $C([0,\infty),\R^o)$ is the sample space for the observed process $y$, and the set $\Omega=\{1,\dots,N\}\times C([0,\infty),\R^o)$ of sample pairs $\omega=(j,\alpha)$ is the sample space for the model. Under uniform convergence on compact intervals, $\Omega$ is a complete metric space.

Define $y(t):\Omega\to\R^o$ and $\theta:\Omega\to\{1,\dots,N\}$ by $y(t,\omega)=\alpha(t)$, $t\ge0$, and $\theta(\omega)=j$. Let $\Y_t=\sigma(y(s),0\le s\le t)$, $t\ge0$, denote the sigma-algebra on $\Omega$ generated by the observations up to time $t$, and let $\Y$ denote the sigma-algebra generated by $\Y_t$, $t\ge0$.
Let $\B_t=\sigma(\theta,\Y_t)$, $t\ge0$, and $\B=\sigma(\theta,\Y)$. Then $\Y_t\subset\B_t$, $t\ge0$, and $\B$ is the Borel sigma-algebra of the metric space  $\Omega$. A  {\em process} is a Borel map on $[0,\infty)\times\Omega$.

In later sections, when we introduce controls, the processes $z(t)$, $w(t)$, and the statistics $P$ (below) will depend on the control $u$ (and the initial condition $(x,p)$). However, the process $y$ and the random variable $\theta$ will always refer to the fixed quantities defined above.

Let $\F_t$ be $\B_t$ or $\Y_t$.   An {\em $\F_t$ process} is an $\F_t$ progressively measurable \cite{KaS} process. Let $P$ be a probability measure on $\B$.  A process $z(t)$ is {\em continuous under $P$} if the sample paths $z(\cdot,\omega)$ are  continuous for $P$ almost all $\omega$ and right-continuous for all $\omega$.

If $z(t)$ is a $\B_t$ process, there are $\Y_t$ processes $z_j(t)$, $j=1,\dots,N$, such that $z(t)=z_\theta(t)$ on $\Omega$. 

A process $z(t)$ is a {\em signal process under $P$}  if
\begin{equation}\label{sqint5}
\int_0^{t} \left(\max_{1\le j\le N}|z_j(s)|^2\right)\,ds<\infty,\qquad t\ge0,
\end{equation}
almost surely $P$.   A continuous process is a signal process. If $z(t)$ is a signal $\B_t$ process under $P$, there is a continuous $\B_t$ process equal almost surely $P$ to $\int_0^tz(s)\,ds$.

Let $i=\sqrt{-1}$ and let
$$\ell(t;ie,w)=\exp\left(i\inner{e}{w(t)}+\frac12|e|^2t\right),$$
$t\ge0$. A $(P,\F_t)$ {\em Wiener process} is a continuous $\F_t$ process $w(t)$ such that 
 $w(0)=0$ and  $\ell(t;ie,w)$ is  a $(P,\F_t)$ martingale, for all $e\in \R^o$ \cite{KaS}.

Let $p=(p_1,\dots,p_N)\in S^N$ and let $z(t)$ be a $\B_t$ process.  A probability measure $P$ on $\B$ is a {\em statistics corresponding to  $p$ and $z(t)$}  if under $P$,
\begin{enumerate}
\item[S1.] $\theta$ is distributed according to $p$,
\item[S2.] $z(t)$ is a signal process, and
\item[S3.] $y(t)-\int_0^tz(s)\,ds$ is a $\B_t$ Wiener process $w(t)$.
\end{enumerate}
Since $\theta$ is $\B_0$ measurable, (S3) implies $w$ and $\theta$ are independent under $P$. 

Let $W$ be Wiener measure on $\Y$. When there is no signal, $G_ \theta\equiv0$,  we have $P=p\times W$.
Let $\delta_j$ be the distribution on $\{1,\dots,N\}$ equal to one at $j$, $j=1,\dots,N$.

\begin{theorem}
\label{local}
Let $\tau$ be a $\B_t$ stopping time. Then for any statistics $P$ corresponding to $p$ and $z(t)$, there is a statistics $P_\tau$ corresponding to $p$ and $z(t)1_{t<\tau}$ satisfying
$P=P_\tau$ on $\B_{\tau}$.
\end{theorem}

\begin{proof}
This is a rephrasing of standard results \cite[6.1.2]{SV}. For $\alpha\in C([0,\infty),\R^o)$ and $t\ge0$, let $W_{t,\alpha}$ denote Wiener measure pinned at $\alpha$ on $[0,t]$, and let $Q_{t,j,\alpha}=\delta_j\times W_{t,\alpha}$. Define
$$P_\tau(A)=E^P\left(  Q_{\tau,\theta,y}(A) \right),\qquad A\in\B.$$
Then 
$$P_\tau(A\inter B)= E^P\left(Q_{\tau,\theta,y}(A);B \right)$$
for $A\in\B$ and $B\in\B_{\tau}$.  In particular, $P=P_\tau$ on $\B_\tau$  and $P_\tau(\theta=j)=p_j$, $j=1,\dots,N$.
Since $z(t)1_{t<\tau}$ is $\B_{\tau}$ measurable, $z(t)1_{t<\tau}$  is a signal process under $P_\tau$. Let $w_\tau(t)=y(t)-\int_0^{t\wedge\tau} z(s)\,ds$. It remains to be shown that $w_\tau(t)$ is a $(P_\tau,\B_t)$ Wiener process.

Since $w(t)=y(t)-\int_0^tz(s)\,ds$ is a $(P,\B_t)$ Wiener process, $\ell(t;ie,w)$ is a $(P,\B_t)$ martingale, hence $\ell(t\wedge\tau;ie,w_\tau)=\ell(t\wedge\tau;ie,w)$ is a $(P,\B_t)$ martingale.  On the other hand, 
$$\ell(t;ie,w_\tau)\ell(t\wedge\tau;ie,w_\tau)^{-1}=\ell(t;ie,y)\ell(t\wedge\tau;ie,y)^{-1}$$
is a $(Q_{\tau,\theta,y},\B_t)$ martingale. By the martingale splicing lemma \cite[1.2.10]{SV}, $\ell(t;ie,w_\tau)$ is a $(P_\tau,\B_t)$ martingale, thus $w_\tau(t)$ is a $(P_\tau,\B_t)$ Wiener process.
\end{proof}

Let $P$ be a probability measure on $\B$.  If $w(t)$ is a $(P,\B_t)$ Wiener process and $z(t)$ is a signal $\B_t$ process under $P$, let 
\begin{equation}\label{likelihood}
\ell(t;z,w)=\exp\left(\int_0^t\inner{z(s)}{dw(s)}-\frac12\int_0^t|z(s)|^2\,ds\right),
\end{equation}
$t\ge0$, where the first integral is an Ito integral (under $P$) \cite{KaS}. Then $\ell(t;z,w)$ is a continuous $\B_t$ process. 

A process $z(t)$ is {\em local under $P$} if for some $n\ge1$,
\begin{equation}\label{L2}
\int_0^\infty \left(\max_{1\le j\le N}|z_j(t)|^2\right)\,dt\le n,
\end{equation}
almost surely $P$. If $z(t)$ is a local $\B_t$ process, then $I(t;z,w)$,  $\ell(t;z,w)$ are $(P,\B_t)$ martingales, $I(\infty;z,w)$,  $\ell(\infty;z,w)$ exist, and $I(\infty;z,w)$,  $\ell(\infty;z,w)$ have means $0$, $1$ respectively. 

Let $z(t)$, $z'(t)$ be $\B_t$ processes. If $P$ is a statistics corresponding to $p$ and $z(t)$, and $z'(t)-z(t)$  is local under $P$, then, by the Girsanov theorem \cite{KaS}, $dP'/dP=\ell(\infty;z'-z,w)$ with $w(t)=y(t)-\int_0^tz(s)\,ds$ defines statistics $P'$ corresponding to $p$ and $z'(t)$, and $z'(t)-z(t)$  is local under $P'$. 

It follows that $P$ is a statistics corresponding to $p$ and $z(t)$ and $z(t)$ is local under $P$ iff $z(t)$ is local under $Q=p\times W$ and 
\begin{equation}\label{girsanov}
P(A)=E^{Q}(\ell(\infty;z,y);A),\qquad A\in\B.
\end{equation}

Let $r(t)$ be the quantity in \eqref{sqint5}. Every signal process can be localized:

\begin{lemma}
\label{lemmalocal}
Let $z(t)$ be a $\B_t$ process. Then there are $\Y_t$ stopping times $\tau_n$, $n\ge1$, such that for any $P$ under which $z(t)$ 
is a signal process, $z(t)1_{t<\tau_n}$ is local under $P$, $r(\tau_n)=n$ on $\tau_n<\infty$ almost surely $P$,  $n\ge1$,  and
 $P(\tau_n\le T)\to0$ as $n\to\infty$ for all $T>0$.
\end{lemma}

\begin{proof}
Let  
$$
\sigma_n^- = \inf\left\{t\ge0: r(t)\ge n\right\}
\qquad\text{and}\qquad
\sigma_n^+ = \inf\left\{t\ge0: r(t)> n\right\}.
$$
Then $\sigma_n^-\le\sigma_n^+$, $n\ge1$, are increasing in $n$, and, for each $n\ge1$, $\sigma_n^\pm$ are $\Y$ measurable. Moreover, $r(\sigma^-_n)\le n$ on $\sigma_n^-<\infty$, and $r(\sigma_n^-)=n$ if in addition $r(\sigma_n^-+)<\infty$. 
For $n\ge1$,
$$
z_n(t)=
\begin{cases}
z(t),\qquad & \text{if }r(t)\le n, \\
0,\qquad & \text{otherwise},
\end{cases}
$$
 is a $\B_t$ process. Let $z_{j,n}(t)$, $j=1,\dots,N$, be the corresponding $\Y_t$ processes, and let $r_n(t)$ be the quantity in \eqref{sqint5} corresponding to $z_n(t)$. Then $r_n(\infty)\le n$  identically on $\Omega$.  Since $r_n(t)$ is a $\Y_t$ process with continuous sample paths identically on $\Omega$,
$$\tau_n=\inf\left\{t\ge0: r_n(t)\ge n\right\},\qquad n\ge1,$$
 are $\Y_t$ stopping times. 

If $z(t)$ is a signal process under $P$, then $r(t)< n$ for $t<\sigma_n^-$, $\sigma_n^-\to\infty$ as $n\to\infty$, and $\tau_n=\sigma_n^-$ almost surely $P$, hence $P(\tau_n\le T)\to0$ as $n\to\infty$, for all $T>0$, and $z(t)1_{t<\tau_n}$ is local under $P$. 
\end{proof}

We call $\tau_n$, $n\ge1$, a {\em localizing sequence for $z(t)$}. 
More generally, if $\tau$ is a $\Y_t$ stopping time and $P$ is such that $r(t)<\infty$ for $t<\tau$ and $r(\tau)=\infty$ on $\tau<\infty$, almost surely $P$, then the proof shows $z(t)1_{t<\tau_n}$ is local under $P$, $r(\tau_n)=n$ on $\tau_n<\infty$ and $\tau_n\le\tau$, both almost surely $P$, and $P(\tau_n\le T<\tau)\to0$ as $n\to\infty$ for all $T>0$. 

That $\tau_n$ is a $\Y_t$ stopping time, rather than a $\B_t$ stopping time, is crucial for the derivation of Theorem \ref{Qcapprox} and the filtering equations in the next section. That $\tau_n$ depends only on $z(t)$, and not $P$, is crucial for

\begin{theorem}
\label{Puniqueness}
There is at most one statistics $P$ corresponding to $p$ and $z(t)$.
\end{theorem}
 
\begin{proof}
Suppose $P$ and $P'$ are statistics corresponding to $p$ and $z(t)$, and let $P_{\tau_n}$ and $P'_{\tau_n}$ be as in Theorem \ref{local}, corresponding to $p$ and $z(t)1_{t<\tau_n}$. Then,  by \eqref{girsanov}, $P_{\tau_n}=P'_{\tau_n}$, and $A\in \B_t$ implies $A_n=A\inter\{t<\tau_n\}\in\B_{\tau_n}$, so
$$P(A_n)=P_{\tau_n}(A_n)=P_{\tau_n}'(A_n)=P'(A_n).$$
Now send $n\to\infty$.
\end{proof}

By similar techniques, one can show the following. Let $\tau$, $\tau'$ be stopping times with $P$, $P'$ statistics corresponding to $p$ and $z(t)1_{t<\tau}$, $z(t)1_{t<\tau'}$ respectively. If $P'(\tau\le \tau')=1$, then $P'=P$ on $\B_{\tau}$. In particular, this implies: If $P_\tau$ and $P$ are statistics corresponding to $z(t)1_{t<\tau}$ and $z(t)$ respectively, then $P=P_\tau$ on $\B_\tau$.

Turning to existence, let $p\in S^N$ and let $z(t)$ be a $\B_t$ process.  If $z(t)$ is local under $Q=p\times W$, by \eqref{girsanov} there is a unique statistics $P$ corresponding to $p$ and $z(t)$, with $P$ sharing the same null events with $Q$ in $\B$.  

Now assume for each $n\ge1$, $\tau_n$ is a stopping time and $P_n$ is a statistics corresponding to $p$ and $z(t)1_{t<\tau_n}$. Assume further the sequence is {\em almost surely increasing} in the sense
\begin{equation}\label{asi}
P_{n'}(\tau_n\le\tau_{n'})=1,\qquad n<n'.
\end{equation}
Then $P_n$, $n\ge1$, are consistently defined on $\B_{\tau_n}$, $n\ge1$.

\begin{theorem}
\label{Pexistence}
If
\begin{equation}\label{taut}
P_n(\tau_n\le t)\to0\qquad\text{as}\qquad n\to\infty,
\end{equation}
for all $t\ge0$, then there is a statistics $P$ corresponding to $p$ and $z(t)$.
\end{theorem}

\begin{proof}
This is a rephrasing of standard results.
Since $P_n$ are consistently defined on $\B_{\tau_n}$, there is a probability measure $P$ on $\B$ whose restriction to $\B_{\tau_n}$ is $P_n$  \cite[1.3.5]{SV}. It follows that $\theta$ is distributed according to $p$ under $P$, $P(\tau_n\le t)\to0$ as $n\to\infty$, and $z(t)$ is a signal process under $P$.

Let $e\in \R^o$.
Since $w_n(t)=y(t)-\int_0^tz_n(s)\,ds$ is a $(P_n,\B_t)$ Wiener process,  $\ell(t;ie,w_n)$ is a $(P_n,\B_t)$ martingale, hence $\ell(t\wedge\tau_n;ie,w)=\ell(t\wedge\tau_n;ie,w_n)$ is a $(P_n,\B_t)$ martingale. Since $P=P_n$ on $\B_{\tau_n}$, $\ell(t\wedge\tau_n;ie,w)$ is a $(P,\B_t)$ martingale. Now send $n\to\infty$.
\end{proof}

We use \eqref{taut} in \S\ref{stabiliz} and \S\ref{optimal} to establish existence of statistics corresponding to specific feedback controls.

Let $e(t)$ be a signal $\B_t$ process under a statistics $P$ corresponding to $p$ and $z(t)$. It is natural to define
\begin{equation}\label{itoPQ}
I(t;e,y)=\int_0^t\inner{e(s)}{dy(s)}=
\int_0^t\inner{e(s)}{z(s)}ds+\int_0^t \inner{e(s)}{dw(s)},
\end{equation}
$t\ge0$.
Then $I(t;e,y)$ and consequently $\ell(t;e,y)$ are continuous $\B_t$ processes.  By Ito's Lemma \cite{KaS}, it follows that
\begin{equation}\label{likeli}
\ell(t;e,y)=1+\int_0^t\ell(s;e,y)\inner{e(s)}{dy(s)},\qquad t\ge0,
\end{equation}
almost surely $P$. 
When $e(t)$ is a $\Y_t$ process, we can do better.

\begin{theorem}
\label{Qcapprox}
Let $P$ be a statistics corresponding to $p$ and $z(t)$. If $e(t)$ is a signal $\Y_t$ process under $P$, there is a continuous $\Y_t$ process $I(t;e,y)$ satisfying \eqref{itoPQ} almost surely $P$.
\end{theorem}

\begin{proof}
Assume first $z(t)$ is local under $P$. By the Girsanov theorem, the probability measure $Q=p\times W$ shares the same null events with $P$ in $\B$, and $y(t)$ is a $(Q,\B_t)$ Wiener process. 
Then $e(t)$ is a signal process under $Q$, hence the $Q$ Ito integral $I(t;e,y)=\int_0^t\inner{e(s)}{dy(s)}$ is a continuous $\Y_t$ process satisfying \eqref{itoPQ} on $t\ge0$, almost surely. 

By localizing, for each $n\ge1$, there is a continuous $\Y_t$ process $I_n(t;e,y)$ equal to \eqref{itoPQ} on $0\le t<\tau_n$, almost surely $P$.
Since for $n<n'$, $I_n(t;e,y)=I_{n'}(t;e,y)$ on $0\le t<\tau_n$, almost surely $P$, by the completeness lemma \cite[4.3.3]{SV}, there is a continuous $\Y_t$ process $I(t;e,y)$ equal to $I_n(t;e,y)$ on $0\le t<\tau_n$, hence equal to \eqref{itoPQ} for $t\ge0$, almost surely $P$. 
 \end{proof}
 
 \begin{lemma}
 \label{lemmadeltaj}
There is a statistics $P$ corresponding to $p$ and $z(t)$ iff there is a statistics $P_j$ corresponding to $\delta_j$ and $z_j(t)$, for each $j$ satisfying $p_j>0$, in which case 
\begin{equation}\label{sum}
P(A)=\sum_{j=1}^N p_jP_j(A)
\end{equation}
and $P(A\mid \theta=j)= P_j(A)$, for $A\in\B$.
\end{lemma}

\begin{proof}
Let $\tau_n$, $n\ge1$, be a localizing sequence for $z(t)$, let $P_n$ be the statistics corresponding to $p$ and $z(t)1_{t<\tau_n}$, and let $P_{n,j}$ be the statistics corresponding to $\delta_j$ and $z_j(t)1_{t<\tau_n}$, for each $j=1,\dots,N$. Since
$$
P_n(\tau_n\le T) = E^{p\times W}(\ell(\tau_n;z,y);\tau_n\le T) 
= \sum_{j=1}^N p_jP_{n,j}(\tau_n\le T),
$$
by Theorem \ref{Pexistence}, $P$ exists iff $P_j$ exist for all $j$ satisfying $p_j>0$. For $A\in\B_T$ let $A_n=A\inter\{T<\tau_n\}\in\B_{T\wedge\tau_n}$. Since
$$
P_n(A_n) = E^{p\times W}(\ell(\tau_n;z,y);A_n) 
=\sum_{j=1}^N p_j P_{n,j}(A_n),
$$
\eqref{sum} follows.
\end{proof}

\section{Filtering Equations}
\label{filtering}

Here we derive the filtering equations, without any integrability or moment conditions on the signal $z(t)=z_\theta(t)$.  For background, see \cite{MHAD, FP, KV, KO, PR}.

Let $P$ be statistics corresponding to $p\in S^N$ and a $\B_t$ process $z(t)$. Then  $\int_0^t\inner{z_j(s)}{dy(s)}$ and the likelihoods
$$\ell_j(t)=\exp\left(\int_0^t\inner{z_j(s)}{dy(s)}-\frac12\int_0^t|z_j(s)|^2\,ds\right),$$
$j=1,\dots,N$, are continuous $\Y_t$ processes, and  $\ell(t)\equiv\ell(t;z,y)=\ell_\theta(t)$, $t\ge0$, almost surely.

Let
\begin{equation}\label{BR}
p_j(t)=\frac{\ell_j(t)p_j}{\ell_1(t)p_1+\dots+\ell_N(t)p_N},\qquad t\ge0,
\end{equation}
$j=1,\dots,N$, and let $\hat z(t)$ be given by \eqref{zhat}.
Then  $\hat z(t)$ is a $\Y_t$ process and $p(t)=(p_1(t),\dots,p_N(t))$ is a continuous $\Y_t$ process.

Let $\tau_n$, $n\ge1$, be a localizing sequence for $z(t)$. If $p_n(t)$ corresponds to $z(t)1_{t<\tau_n}$, then $p_n(t\wedge\tau_n)=p(t\wedge\tau_n)$, $t\ge0$, almost surely.

We show $p_j(t)$, as defined by \eqref{BR}, is the conditional probability that $\theta=j$ given $\Y_t$. Then \eqref{BR} is Bayes' rule.

\begin{theorem}
\label{theorem2} 
For $t\ge0$, $j=1,\dots,N$, and $A\in\Y_t$, 
\begin{equation}\label{cpj}
P(\{\theta=j\}\inter A)=E^P(p_j(t);A),
\end{equation}
and \eqref{condprob} is valid.
\end{theorem}

\begin{proof} 
The validity of \eqref{cpj} for all $A\in\Y_t$ is the definition of \eqref{condprob}. For \eqref{cpj}, let  $\Phi$ be a nonnegative $\Y_t$ measurable random variable. 
Since $\theta$ and $y(t)$ are independent under $p\times W$, 
\begin{equation*}
\begin{split}
E^{p\times W}(\ell(t)\Phi;\theta=j) &=  E^{p\times W}(\ell_\theta(t)\Phi;\theta=j) \\
&=  E^{p\times W}(\ell_j(t)\Phi;\theta=j) \\
&=  E^{p\times W}(\ell_j(t)p_j\Phi),
\end{split}
\end{equation*}
hence
$$E^{p\times W}(\ell(t)\Phi)=E^{p\times W}\left(\sum_{k=1}^N\ell_k(t)p_k\Phi\right).$$
If $z(t)$ is local under $P$,  applying these with $\Phi=1_A$ and $\Phi=p_j(t)1_A$, by \eqref{girsanov},
\begin{equation*}
\begin{split}
P(\{\theta=j\}\inter A) &=   E^{p\times W}(\ell(t);\{\theta=j\}\inter A) \\
&=  E^{p\times W}(\ell_j(t)p_j; A) \\
&=   E^{p\times W}(\ell(t)p_j(t);A) \\
&=  E^{P}(p_j(t); A).
\end{split}
\end{equation*}
If $z(t)$ is not local under $P$, let $A_n=A\inter\{t<\tau_n\}$ and localize to get
\begin{equation*}
\begin{split}
P(\{\theta=j\}\inter A_n) 
&= P_n(\{\theta=j\}\inter A_n) \\
&= E^{P_n}(p_{j,n}(t); A_n) \\
&= E^{P_n}(p_{j}(t); A_n) \\
&= E^{P}(p_{j}(t); A_n),
\end{split}
\end{equation*}
then send $n\to\infty$.
\end{proof}

Since $p(t)$ is a bounded  $(P,\Y_t)$  martingale, $p(\infty)$ exists almost surely $P$ \cite{KaS}.

\begin{theorem}
\label{theorem3}
Let $\hat z(t)$ be given by \eqref{zhat}. Then 
\begin{equation}\label{inno}
\nu(t)=y(t)-\int_0^t \hat z(s)\,ds,\qquad t\ge0,
\end{equation}
is a $(P,\Y_t)$ Wiener process.
\end{theorem}

\begin{proof} 
Since
$$\nu(t)=w(t)+\int_0^t(z(s)-\hat z(s))\,ds,$$
almost surely $P$, by Ito's Lemma \cite{KaS}, $\ell(t;ie,\nu)$ satisfies
$$
\ell(t;ie,\nu) = \ell(s;ie,\nu) 
+ i\int_s^t\ell(r;ie,\nu)\inner{e}{dw(r)} 
+ i\int_s^t\ell(r;ie,\nu)\inner{e}{z(r)-\hat z(r)}\,dr,
$$
almost surely $P$. Taking the expectation of both sides over $A\in\Y_s$,
since $w$ is a $(P,\B_t)$ Wiener process, the expectation of the Ito integral vanishes [the integrand is local on finite time intervals].
Since  $A\in\Y_{r}$ and $\ell(r;ie,\nu)$ is $\Y_r$ measurable,  the expectation of the second integral also vanishes. Thus
$$E^P(\ell(t);ie,\nu) ;A)=E^P(\ell(s);ie,\nu);A)$$
for $A\in\Y_s$. This shows $\ell(t;ie,\nu)$ is a $(P,\Y_t)$ martingale for $e\in \R^o$, hence $\nu(t)$ is a $(P,\Y_t)$ Wiener process. 
\end{proof}

It follows that for any signal $\Y_t$ process $e(t)$, 
$$\int_0^t\inner{e(s)}{d\nu(s)}= \int_0^t\inner{e(s)}{dw(s)}+ \int_0^t\inner{e(s)}{z(s)-\hat z(s)}\,ds,$$
$t\ge0$, almost surely $P$.

\begin{theorem}
\label{theorem4}
Under $P$, the conditional probabilities satisfy
\begin{equation}\label{filter}
p_j(t)=p_j+\int_0^tp_j(s)\inner{z_j(s)-\hat z(s)}{d\nu(s)},
\end{equation}
$t\ge0$, $j=1,\dots,N$, almost surely $P$.
\end{theorem}

\begin{proof} 
Apply Ito's Lemma to the processes $\ell_j(t)$ and $p_j(t)$ driven by the $(P,\B_t)$ Wiener 
process $w(t)$, together with  \eqref{likeli}. 
\end{proof}

Let $H$ be as in \eqref{calculusfact}. The following result, which says the time rate of change of the expected information $-E^P(H(p(t)))$, $t\ge0$, is the mean square conditional variance of the signal $z(t)$, holds as soon as statistics $P$ exist for the signal process $z(t)$, with no moment conditions imposed on the signal.

\begin{theorem}
\label{theorem6}
Let $\tau$ be a $\Y_t$ stopping time. Then 
\begin{equation}\label{meansqvar}
H(p)=E^P(H(p(\tau)))+\frac12E^P\left(\int_0^\tau|z(t)-\hat z(t)|^2\,dt\right).
\end{equation}
In particular, this is so for $\tau=\infty$.
\end{theorem}

\begin{proof}
By replacing $\tau$ by $\tau\wedge\tau_n$ and sending $n\to\infty$, we may assume $z(t)$ is local.
With $\L$ is as in \eqref{gen}, \eqref{calculusfact} implies
\begin{equation}\label{d^2entropy}
-\L H=\frac12\sum_{j=1}^Np_j|z_j-\hat z|^2,
\end{equation}
where $\hat z=\sum_{j=1}^N p_jz_j$. 
Apply Ito's Lemma to $H(p(t))$, $0\le t<\tau$,  using  \eqref{filter}. Then the resulting Ito integral has an integrand that is local under $P$, hence its expectation vanishes. This yields 
\begin{equation*}
\begin{split}
H(p)  - E^P(H(p(\tau))) 
&= \frac12E^P\left(\int_0^{\tau}\sum_{j=1}^Np_j(t)|z_j(t)-\hat z(t)|^2\,dt\right) \\
&=  \frac12\sum_{j=1}^N\int_0^{\infty}E^P(p_j(t)|z_j(t)-\hat z(t)|^2;t<\tau)\,dt \\
&= \frac12\sum_{j=1}^N\int_0^{\infty}E^P(|z_j(t)-\hat z(t)|^2;t<\tau,\theta=j)\,dt \\
&= \frac12\int_0^\infty E^P(|z(t)-\hat z(t)|^2;t<\tau)\,dt \\
&= \frac12E^P\left(\int_0^{\tau}|z(t)-\hat z(t)|^2\,dt\right).
\end{split}
\end{equation*}
\end{proof}

\section{Stabilizability}
\label{stabiliz}

A  {\em control} is a $\Y_t$ process $u$ valued in $\R^i$. 

Let $x=(x_1,\dots,x_N)$ and $p=(p_1,\dots,p_N)$. If  there is a probability measure $P$ on $\B$ such that \eqref{L2control} holds almost surely $P$, then there is a continuous $\Y_t$ process $x_j(t)$ satisfying
\begin{equation}\label{xjeq}
\dot x_j(t)= A_j x_j(t)+B_ju(t),\qquad x_j(0)=x_j,
\end{equation}
$j=1,\dots,N$, almost surely $P$.  Then the state and signal processes corresponding to $u$ are $x_\theta(t)$ and $z(t)=z_\theta(t)=G_\theta x_\theta(t)$.

A control is {\em admissible at $(x,p)$} if there is a probability measure $P$ on $\B$ such that \eqref{L2control} holds almost surely $P$, and $P$ is the statistics (\S\ref{pols}) corresponding to $p$ and $z(t)$.  In this case, we call $P$ the {\em statistics corresponding to $(x,p,u)$.}  When there is no signal,  $G_\theta\equiv0$, every control is admissible at $(x,p)$, and $P=p\times W$.

Let $u$ be a control and let $z(t)$, $z'(t)$ be the corresponding signal processes starting from initial state ensembles $x$ and $x'$ respectively. Since
$$z'(t)-z(t)=G_\theta e^{A_\theta t}(x'_\theta-x_\theta),$$
the Girsanov theorem implies there is a statistics corresponding to $p$ and $z(t)1_{t<T}$ iff there is a statistics corresponding to $p$ and $z'(t)1_{t<T}$. Since this is true for all $T>0$, by Theorem \ref{Pexistence}, there is a statistics $P$ corresponding to $p$ and $z(t)$ iff there is a statistics $P'$ corresponding to $p$ and $z'(t)$. Since $P$ and $P'$ share the same null events in $\B_T$, $T\ge0$, \eqref{L2control} holds almost surely $P$ iff it holds almost surely $P'$. Thus $u$ is admissible at $(x,p)$ iff $u$ is admissible at $(x',p)$. This together with  Lemma \ref{lemmadeltaj} implies a control $u$ is admissible at a single $(x,p)$ with $p$ in the interior of $S^N$ iff $u$ is admissible at all $(x,p)$. Thus the class of controls admissible at $(x,p)$, with $p$ in the interior of $S^N$, does not depend on $(x,p)$, and \eqref{affine} follows from Lemma \ref{lemmadeltaj}.

\begin{theorem}
\label{minimal}
Assume $(A_j,B_j,C_j)$, $j=1,\dots,N$, are minimal. Then $J(x,p,u)$ finite implies $u$ is stabilizing at $(x,p)$.
 \end{theorem}

\begin{proof}
\eqref{estimateK} implies 
$$\frac12\int_t^\infty\left(|C_\theta x(s)|^2+|u(s)|^2\right)\,ds \ge \frac12\inner{K_\theta x(t)}{x(t)}$$
almost surely $P$, which implies $x(t)\to0$ as $t\to\infty$, almost surely $P$.
\end{proof}

\begin{theorem}
\label{adapunif}
The system \eqref{stateeq1},\eqref{obser1} is uniformly stabilizable with feedback $F$ if and only if  it is adaptively stabilizable with feedback $F$. If $U^F$ and $V^F$ are the costs associated to the feedback \eqref{cep}, then \eqref{feedbackV} holds.
\end{theorem}

\begin{proof}
If the system is adaptively stabilizable with feedback $F$, then by choosing $p=\delta_j$, $j=1,\dots,N$, we see it is uniformly stabilizable with feedback $F$. Conversely, suppose $\bar A_j=A_j-B_jFG_j$, $j=1,\dots,N$, are stable for some feedback $F$, and fix an initial condition $(x,p)$. 

Under $Q=p\times W$, $y$ is a Wiener process independent of $\theta$, and the system \eqref{zhat}, \eqref{cep}, \eqref{inno}, \eqref{filter}, \eqref{xjeq} of stochastic differential equations driven by $y$ has smooth coefficients, hence there is  a unique continuous $\Y_t$ process \eqref{dyna}
starting from $(x,p)$, satisfying \eqref{zhat}, \eqref{cep}, \eqref{inno}, \eqref{filter}, \eqref{xjeq}, and 
defined up to a $\Y_t$ stopping time $\tau\le\infty$ with \cite{HPM}
$$\max_{1\le j\le N}|x_j(\tau)|^2=\infty,\qquad \text{on }\tau<\infty,$$ 
almost surely $Q$.  

Let $z_j(t)=G_jx_j(t)1_{t<\tau}$, $j=1,\dots,N$, and define $u$ by \eqref{zhat}, \eqref{cep}.  Then $u$ and $z_j(t)$, $j=1,\dots,N$, are $\Y_t$ processes and \eqref{zhat}, \eqref{cep}, \eqref{xjeq} imply $r(\tau)=\infty$ on $\tau<\infty$ almost surely $Q$, where $r(t)$ is the quantity in \eqref{L2control}.

The goal is to show there are statistics $P$ corresponding to $p$ and $z(t)=z_\theta(t)$, with $\tau=\infty$ almost surely $P$.  Because the existence of $P$  depends on $\tau$ being infinite almost surely $P$, and $\tau$ being infinite almost surely presumes the existence of $P$, we seem to be stuck in a circular situation. What saves us is the entropy bound \eqref{meansqvar}.

Let $\tau_n$, $n\ge1$, be a localizing sequence for $u$ and let $r(t)$ be the quantity in \eqref{L2control}.
  Then $\tau_n$ is a $\Y_t$ stopping time such that 
$r(\tau_n)=n$ on $\tau_n<\infty$ and $\tau_n\le\tau$, both almost surely $Q$, and $Q(\tau_n\le T<\tau)\to0$ as $n\to\infty$, for all $T>0$. Fix $T>0$ and let $u_n(t)=u(t)1_{t<T\wedge\tau_n}$, $x_{j,n}(t)=x_j(t)1_{t<T\wedge\tau_n}$, and $z_{j,n}(t)=G_jx_{j,n}(t)$, $j=1,\dots,N$. Then $u_n$ is a control and $x_{\theta,n}(t)$ equals the state process corresponding to $u_n$ starting from $x_\theta$, on $0\le t<T\wedge\tau_n$. Since $z_n(t)=z_{\theta,n}(t)$ is  local under $Q$, there are statistics $P_n$ corresponding to $p$ and $z_n(t)$ such that $P_n$ shares the same null events with $Q$ in $\B$. Since $\tau_n$, $n\ge1$, is almost surely increasing \eqref{asi}, the statistics $P_n$ are consistently defined on $\B_{\tau_n}$, $n\ge1$.

Note
$$E^{P_n}(|x_\theta|^2)=\sum_{j=1}^N p_j|x_j|^2$$
does not depend on $n\ge1$.

Let $p_n(t)=(p_{1,n}(t),\dots,p_{N,n}(t))\in S^N$ be the corresponding conditional probabilities. Then 
$$(x_{1,n}(t),\dots,x_{N,n}(t),p_{1,n}(t),\dots,p_{N,n}(t))$$ 
is a solution of  \eqref{zhat}, \eqref{cep}, \eqref{inno}, \eqref{filter}, \eqref{xjeq} up to time $T\wedge\tau_n$. By uniqueness of solutions of stochastic differential equations, $x_j(t\wedge\tau_n)=x_{j,n}(t\wedge\tau_n)$,  $p_j(t\wedge\tau_n)=p_{j,n}(t\wedge\tau_n)$, $0\le t<T\wedge\tau_n$, almost surely $Q$, hence almost surely $P_n$, for each $j=1,\dots,N$. 

By Theorem \ref{theorem6},
\begin{equation}\label{entropybound}
E^{P_n}\left(\int_0^{T\wedge\tau_n} |z(t)-\hat z(t)|^2\,dt\right)\le H(p).
\end{equation}
With $\bar A_\theta=A_\theta-B_\theta FG_\theta$,  $x(t)=x_\theta(t)$ satisfies
\begin{equation}\label{linear2}
\dot x= A_\theta x+ B_\theta u= \bar A_\theta x+ B_\theta F(z-\hat z)
\end{equation}
for $t<T\wedge\tau_n$, hence
\begin{equation}\label{convo}
x(t)=e^{\bar A_\theta t}x_\theta+\int_0^te^{\bar A_\theta(t-s)}B_\theta F(z(s)-\hat z(s))\,ds
\end{equation}
for $t<T\wedge\tau_n$, almost surely $P_n$. 
Since $\bar A_\theta$ is stable, there is a $\mu>0$ such that $\Re \bar A_j\le -\mu$ and $1/\mu$ bounds $|B_j|$ and $|F|$. Then by Young's convolution inequality \cite{LL}, and \eqref{entropybound}, 
\begin{equation}\label{est1}
 E^{P_n}\left(\int_0^{T\wedge\tau_n}|x(t)|^2\,dt\right) \le c_1(\mu) \left(E^{P_n}(|x_\theta|^2)+ H(p)\right)
\end{equation}
for some constant $c_1(\mu)$ depending only on $\mu$. By \eqref{cep}, this implies
\begin{equation}\label{est2}
 E^{P_n}(r(T\wedge\tau_n)) \le c_2(\mu) \left(E^{P_n}(|x_\theta|^2)+ H(p)\right),
\end{equation}
where now $|G_j|\le1/\mu$ as well. Since  $r(\tau_n)=n$ on $\tau_n<\infty$ almost surely $Q$ hence almost surely $P_n$, 
\begin{equation}
\label{last}
nP_n(\tau_n\le T) = E^{P_n}(r(T\wedge\tau_n);\tau_n\le T) \\
\le c_2(\mu) \left(E^{P_n}(|x_\theta|^2)+ H(p)\right). 
\end{equation}
This proves \eqref{taut}, hence there are statistics $P$ corresponding to $p$ and $z(t)$ whose restriction to $\B_{\tau_n}$ is $P_n$. By \eqref{last},  $\tau=\infty$ almost surely $P$, hence $u$ is an admissible control. Replacing $P_n$ by $P$ in \eqref{est1}, \eqref{est2} and sending $n\to\infty$ then $T\to\infty$  implies \eqref{feedbackV}, where now $|C_j|\le1/\mu$ as well, hence $u$ is stabilizing at $(x,p)$. Uniqueness of the admissible control follows from uniqueness of solutions of the system \eqref{zhat}, \eqref{cep}, \eqref{inno}, \eqref{filter}, \eqref{xjeq}.
\end{proof}

\section{Explicit Solutions}
\label{optimal}

Given $f=f(x_1,\dots,x_N,p_1,\dots,p_N)$, let
\begin{equation}\label{Value1}
V^f(x,p)=\inf_u \left(\vphantom{\int}J(x,p,u)+E^P(f(0,p(\infty)))\right),
\end{equation}
where the infimum is over controls admissible at $(x,p)$.
When $f(0,p)=0$, $V^f=U$, and when $f(0,p)=H(p)$, $V^f=V$.
We derive a verification theorem suitable for the explicit supersolutions.

\begin{theorem}
\label{super}
Let $f$ be a $C^2$ nonnegative supersolution of the Bellman equation \eqref{bellman1}.
Then for each initial $(x,p)$, there is a unique admissible control $u$, with statistics $P$, satisfying
\begin{equation}\label{cep2}
u(t)=-\sum_{j=1}^NB_j^*\nabla_jf(x(t),p(t)),\qquad t\ge0,
\end{equation}
almost surely $P$, with corresponding cost
\begin{equation}\label{VU}
V^f(x,p)\le J(x,p,u)+E^P(f(0,p(\infty)))\le f(x,p).
\end{equation}
If $f$ is a $C^2$ nonnegative solution of \eqref{bellman1}, then for each $(x,p)$, $f(x,p)=V^f(x,p)$
 and the admissible control satisfying \eqref{cep2} is optimal at $(x,p)$.
\end{theorem}

\begin{proof}  
Given $u\in\R^i$, let
$$\L^uf=\L f + \sum_{j=1}^N\inner{\nabla_j f}{A_jx_j+B_ju}.$$
Let $u$ be any admissible control and let \eqref{dyna} be the corresponding processes satisfying \eqref{filter} and \eqref{xjeq}.
 Ito's Lemma applied to $f(t)\equiv f(x(t),p(t))$, $t\ge0$, yields
$$
E^{P}(f(\sigma))=f(x,p) + E^P\left(\int_0^{\sigma} \L^u f\,dt\right)
$$
for any $\Y_t$ stopping time $\sigma$ with $|x(t)|\le c$, $0\le t<\sigma$.
Let
$$
J_\sigma(x,p,u)=\frac12E^{P}\left(\int_0^{\sigma} (|u|^2+|C_\theta x|^2)\,dt\right).
$$
Since $f$ is a supersolution of \eqref{bellman1}, adding the last two equations implies 
\begin{equation}\label{ineq3} 
J_\sigma(x,p,u)  + E^{P}(f(\sigma)) 
\le 
f(x,p) + \frac12E^{P}\left(\int_0^{\sigma} \left|u+\sum_{j=1}^NB_j^*\nabla_jf\right|^2\,dt\right).  
\end{equation}
The rest of the proof is a repetition of that of Theorem \ref{adapunif}.  Under $Q=p\times W$, the system  \eqref{zhat}, \eqref{inno}, \eqref{filter}, \eqref{xjeq}, \eqref{cep2} of stochastic differential equations driven by $y$ has smooth coefficients, hence there is  a unique continuous $\Y_t$ process \eqref{dyna}
starting from $(x,p)$, satisfying  \eqref{zhat}, \eqref{inno}, \eqref{filter}, \eqref{xjeq},  \eqref{cep2} and 
defined up to a $\Y_t$ stopping time $\tau\le\infty$ with \cite{HPM}
$$\max_{1\le j\le N}|x_j(\tau)|^2=\infty,\qquad \text{on }\tau<\infty,$$ 
almost surely $Q$. 

Let $z_j(t)=G_jx_j(t)1_{t<\tau}$, $j=1,\dots,N$, and define $u(t)$ by \eqref{cep2} on $0\le t<\tau$ and $u(t)=0$ otherwise. Then $u(t)$ and $z_j(t)$, $j=1,\dots,N$, are $\Y_t$ processes and \eqref{zhat}, \eqref{cep}, and \eqref{xjeq} imply $r(\tau)=\infty$ on $\tau<\infty$ almost surely $Q$, where $r(t)$ is the quantity in \eqref{L2control}.

Let $\tau_n$, $n\ge1$, be a localizing sequence for $u(t)$.  
Then $\tau_n$ is a $\Y_t$ stopping time such that 
$r(\tau_n)=n$ on $\tau_n<\infty$ and $\tau_n\le\tau$,  almost surely $Q$, and $Q(\tau_n\le T<\tau)\to0$ as $n\to\infty$. 
 Using $f\ge0$ and \eqref{cep2}, \eqref{ineq3} with $\sigma=T\wedge\tau_n$,
$$\frac12E^{P_n}(r(T\wedge\tau_n))\le J_{T\wedge\tau_n}(x,p,u)\le f(x,p)$$
which implies  as before $u$ is an admissible control with corresponding statistics $P$ and satisfying \eqref{cep2}. Sending $n\to\infty$ then $T\to\infty$  in \eqref{ineq3} with  $\sigma=T\wedge\tau_n$ implies $J(x,p,u)\le f(x,p) <\infty$, hence $x(\infty)=0$,  hence  \eqref{VU}.

For the second statement, if $f$ is a solution of \eqref{bellman1}, then for any admissible control $u$,  \eqref{ineq3}  is an equality for all $\sigma$. Inserting $\sigma=\infty$ yields $J(x,p,u)  + E^{P}(f(0,p(\infty))) \ge f(x,p)$ for all $u$ which  implies $f=V^f$, and $u$ given by \eqref{cep2} is optimal at $(x,p)$. 
 \end{proof}

Given vectors $\upsilon_j$, $j=1,\dots,N$, and $p\in S^N$, let $\hat{\upsilon}=\sum_{j=1}^Np_j\upsilon_j$. Then
\begin{equation}
\label{variance}
\sum_{j=1}^N p_j|\upsilon_j-\hat \upsilon|^2=\sum_{j=1}^N p_j|\upsilon_j|^2-|\hat \upsilon|^2.
\end{equation}

Let $\upsilon_j=B^*_jK_jx_j$, $j=1,\dots,N$.
Inserting $f=U_{ce}$ into \eqref{bellman1} and using \eqref{algricatti}, we see $U_{ce}$ is a subsolution or supersolution of \eqref{bellman1} if and only if \eqref{variance}
is $\ge0$, for all $x_1,\dots,x_N,p_1,\dots,p_N$, or $\le0$, for all $x_1,\dots,x_N,p_1,\dots,p_N$. It follows that $U_{ce}$ is always a subsolution. If $U_{ce}$ is a supersolution and $N\ge2$, then setting $x_k=0$ for all $k\not=j$ yields $p_j(1-p_j)|B_j^*K_jx_j|^2=0$, hence $B_j^*K_j=0$. But this contradicts minimality of $(A_j,B_j,C_j)$. Hence $U_{ce}$ is never a supersolution unless $N=1$.

\begin{lemma}
\label{ineq2}
Let $G_j$ and $H_j$, $j=1,\dots,N$, be linear maps into $\R^o$ and $\R^i$ respectively. If there exists a linear map $F:\R^o\to\R^i$ of norm at most one such that $H_j=FG_j$, $j=1,\dots,N$, then
\begin{equation}\label{ineq}
\sum_{j=1}^Np_j\left|G_jx_j-\widehat{Gx}\right|^2 \ge \sum_{j=1}^Np_j\left|H_jx_j-\widehat{Hx}\right|^2
\end{equation}
for all $p\in S^N$ and all $x_j$, $j=1,\dots,N$. Conversely, if $N\ge2$ and the linear map $G_1$ is surjective, and \eqref{ineq} holds, there exists a linear map $F:\R^o\to\R^i$ of norm at most one such that $H_j=FG_j$, $j=1,\dots,N$. If equality holds in \eqref{ineq}, then $F:\R^o\to\R^i$ may be chosen a partial isometry.
\end{lemma}

\begin{proof}
If $H_j=FG_j$, $j=1,\dots,N$, with the norm of $F$ at most one, then \eqref{ineq} is immediate. Conversely, assume \eqref{ineq}. Choosing $j$ and $x_k=0$ for $k\not=j$ implies $|G_jx_j|^2\ge |H_jx_j|^2$ hence there exists $F_j$ on $\R^o$ such that $H_j=F_jG_j$, $j=1,\dots,N$. 

Choosing $j$ and $p_1+p_j=1$, \eqref{ineq} yields $|G_jx_j-G_1x_1|^2\ge |F_jG_jx_j-F_1G_1x_1|^2$. Choosing $x_1$ such that $G_1x_1=G_jx_j$  implies $F_1G_jx_j=F_jG_jx_j=H_jx_j$, hence $F=F_1$.
\end{proof}

\begin{theorem}
$V_{ce}$ is a supersolution of the Bellman equation if there exists $F:\R^o\to\R^i$ of norm at most one satisfying \eqref{explicit}. Conversely, if $N\ge2$ and the signal is faithful and $V_{ce}$ is a supersolution, there exists $F:\R^o\to\R^i$ of norm at most one satisfying \eqref{explicit}. When this happens, the unique admissible control $u$ given by \eqref{cep} satisfies \eqref{VV}, and the system is adaptively stabilizable with feedback $F$. 
\end{theorem}

\begin{proof}
Inserting \eqref{cep1} into \eqref{bellman1} and using \eqref{algricatti}, \eqref{d^2entropy}, we see $V_{ce}$ is a supersolution of \eqref{bellman1} if and only if 
\begin{equation}\label{var}
\sum_{j=1}^Np_j\left|G_jx_j-\widehat{Gx}\right|^2 \ge \sum_{j=1}^Np_j\left|B^*_jK_jx_j-\widehat{B^*Kx}\right|^2,
\end{equation}
holds for all $x_1,\dots,x_N,p_1,\dots,p_N$. If  there is a feedback $F:\R^o\to\R^i$ of norm at most one such that \eqref{explicit} holds, then \eqref{var} holds. Conversely, if \eqref{var} holds, then by Lemma \ref{ineq2}, there is an $F:\R^o\to\R^i$ of norm at most one such that \eqref{explicit} holds. The second statement follows from Theorem \ref{super}.
\end{proof}

\begin{theorem}
Certainty equivalence holds if  there exists a partial isometry $F:\R^o\to\R^i$ satisfying \eqref{explicit}. Conversely, if $N\ge2$ and the signal is faithful and certainty equivalence holds, there exists a partial isometry $F:\R^o\to\R^i$ satisfying \eqref{explicit}. When this happens, the cost of the feedback \eqref{cep} starting from $(x,p)$ equals $V_{ce}(x,p)$, and the system is adaptively stabilizable with feedback $F$. 

\end{theorem}

\begin{proof}
Same proof as the previous Theorem except \eqref{var} is now an equality.
\end{proof}

\section*{Disclosure statement}

No potential conflict of interest was reported by the author.


\begin{thebibliography}{10}
\providecommand{\MR}{\relax\unskip\space MR }
\providecommand{\url}[1]{\normalfont{#1}}
\providecommand{\urlprefix}{Available at }

\bibitem{AK}
K.J. \AA{}str{\"o}m and P.R. Kumar, \emph{Control: A perpective}, Automatica 50
  (2014), pp. 3--43.

\bibitem{AM}
B.D.O. Anderson and J.B. Moore, \emph{Linear Optimal Control}, Prentice-Hall,
  1971.

\bibitem{YBS}
Y. Bar-Shalom and E. Tse, \emph{Dual effect, certainty equivalence, and
  separation in stochastic control}, IEEE Transactions on Automatic Control 19
  (1974), pp. 494--500.

\bibitem{DPB}
D.P. Bertsekas, \emph{Dynamic Programming and Optimal Control (4th ed.)},
  Athena Scientific, 2017.

\bibitem{BGW}
R. Bitmead, M. Gevers, and V. Wertz, \emph{Adaptive Optimal Control},
  Prentice-Hall, 1990.

\bibitem{RWB}
R.W. Brockett, \emph{Finite Dimensional Linear Systems}, Classics in Applied
  Mathematics, SIAM, 2015.

\bibitem{CTC}
C.T. Chen, \emph{Linear System Theory and Design, Third Edition}, Oxford
  University Press, 1999.

\bibitem{CIL}
M.G. Crandall, H. Ishii, and P.L. Lions, \emph{User's guide to viscosity
  solutions of second order partial differential equations}, Bull. Amer. Math.
  Soc. (N.S.) 27 (1992), pp. 1--67.

\bibitem{MHAD}
M.H.A. Davis, \emph{Nonlinear filtering and stochastic flows}, in \emph{Proc.
  {I}nt. {C}ongr. {M}ath.}, Berkeley, CA. IMU, 1986.

\bibitem{F}
A.A. Fel'dbaum, \emph{Optimal Control Systems}, Academic Press, 1965.

\bibitem{FP}
W.H. Fleming and E. Pardoux, \emph{Optimal control for partially observed
  diffusions}, SIAM Journal on Control and Optimization 20 (1982).

\bibitem{FS}
W.H. Fleming and H.M. Soner, \emph{Controlled Markov Processes and Viscosity
  Solutions}, Stochastic Modelling and Applied Probability Vol.~25,
  Springer-Verlag New York, 2006.

\bibitem{GS}
G.C. Goodwin and K.S. Sin, \emph{Adaptive Filtering Prediction and Control},
  Dover, 2009.

\bibitem{OH}
O. Hijab, \emph{Stabilization of Control Systems}, Springer, 1987.

\bibitem{IS}
P. Ioannou and J. Sun, \emph{Robust Adaptive Control}, Courier Corporation,
  2013.

\bibitem{KI}
K. Ito, \emph{Existence of solutions to the {HJB} equation under quadratic
  growth conditions}, J. Diff. Eqs. 176 (2001), pp. 1--28.

\bibitem{REK}
R.E. Kalman, \emph{Contributions to the theory of optimal control}, Boletin de
  la Sociedad Matematica Mexicana 5 (1960), pp. 102--119.

\bibitem{KaS}
I. Karatzas and S.E. Shreve, \emph{Brownian Motion and Stochastic Calculus},
  Graduate Texts in Mathematics, Springer, 2012.



\bibitem{AJK}
A.J. Krener, \emph{Stochastic {HJB} Equations and Regular Singular
  Points}, arXiv e-prints,  (2018), arXiv:1806.04120.
  
\bibitem{KKK}
M. Krstic, I. Kanellakopoulos, and P. Kokotovic, \emph{Adaptive Control:
  Stability, Convergence and Robustness}, John Wiley and Sons, 1995.

\bibitem{KV}
P.R. Kumar and P. Varaiya, \emph{Stochastic Systems: Estimation,
  Identification, and Adaptive Control}, Classics in Applied Mathematics, SIAM,
  2016.

\bibitem{KO}
T.G. Kurtz and D.L. Ocone, \emph{Unique characterization of conditional
  distributions in nonlinear filtering}, Ann. Prob. 16 (1988), pp. 80--107.

\bibitem{LL}
E.H. Lieb and M. Loss, \emph{Analysis}, American Mathematical Society, 2001.

\bibitem{HPM}
H.P. McKean, \emph{Stochastic Integrals}, Probability and mathematical
  statistics, Academic Press, 1969.

\bibitem{NA}
K.S. Narendra and A.M. Annaswamy, \emph{Stable Adaptive Systems}, Dover, 2005.

\bibitem{PR}
E. Pardoux and A. R\v{a}\c{s}canu, \emph{Stochastic Differential Equations,
  Backward SDEs, Partial Differential Equations}, Springer, 2014.

\bibitem{SB}
S. Sastry and M. Bodson, \emph{Nonlinear and Adaptive Control Design}, Dover,
  1989.

\bibitem{S}
E.D. Sontag, \emph{Mathematical Control Theory}, Texts in Applied Mathematics
  Vol.~6, Springer-Verlag New York, 1998.

\bibitem{SV}
D.W. Stroock and S.R.S. Varadhan, \emph{Multidimensional diffusion processes},
  Classics in Mathematics, Springer Verlag Berlin, New York, 1979.

\bibitem{T}
G. Tao, \emph{Adaptive Control Design and Analysis}, John Wiley and Sons, 2003.

\end{thebibliography}

\end{document}